\documentclass[a4paper,draft,12pt,reqno]{amsart}
\usepackage{amsmath,amssymb,amsthm,xspace,amscd,latexsym,amssymb,stmaryrd}
\usepackage[mathscr]{eucal}
\usepackage[all]{xy}
\usepackage{enumerate}

\theoremstyle{plain}
\newtheorem{theorem}{Theorem}
\newtheorem{remark}[theorem]{Remark}
\newtheorem{lemma}[theorem]{Lemma}
\newtheorem{proposition}[theorem]{Proposition}
\newtheorem{corollary}[theorem]{Corollary}

\newtheorem{problem}[theorem]{Problem}

\newcommand{\tto}{\twoheadrightarrow}

\DeclareMathOperator{\Hom}{Hom}

\begin{document} 

\title{Some homological properties of the category $\mathcal{O}$, II}
\author{Volodymyr Mazorchuk}
\date{}
\maketitle

\begin{abstract}
We show, in full generality, that Lusztig's $\mathbf{a}$-function
describes the projective dimension of both indecomposable tilting
modules and indecomposable injective modules in the regular block 
of the BGG category $\mathcal{O}$, proving a conjecture from the 
first paper. On the way we show that the images of simple modules 
under projective functors can be represented in the derived category 
by linear complexes of tilting modules. These complexes, in turn, can
be interpreted as the images of simple modules under projective
functors in the Koszul dual of the category $\mathcal{O}$.
Finally, we describe the dominant projective module and also
projective-injective modules in some subcategories of
$\mathcal{O}$ and show how one can use categorification to decompose the 
regular representation of the Weyl group into a direct sum of 
cell modules, extending the results known for the symmetric group
(type $A$).
\end{abstract}
\vspace{2mm}

\noindent
{\bf 2000 Mathematics Subject Classification:} 16E10; 16E30; 16G99; 17B10
\vspace{2mm}

\noindent
{\bf Keywords:} category $\mathcal{O}$; tilting module; Lusztig's 
$\mathbf{a}$-function; complex; Koszul dual; categorification;
projective dimension
\vspace{2mm}

\section{Introduction}\label{s0}

Let $\mathfrak{g}$ be a semi-simple complex finite-dimensional Lie algebra
and $\mathcal{O}_0$ the principal block of the BGG category $\mathcal{O}$
for $\mathfrak{g}$ (\cite{BGG}). After the Kazhdan-Lusztig conjecture,
formulated  in \cite{KL}  and proved in \cite{BB,BK}, it became clear that 
many algebraic properties of $\mathcal{O}_0$ can be studied using the
Kazhdan-Lusztig combinatorics (\cite{KL,BjBr}). In the first paper
\cite{Ma2} I conjectured that the projective dimension of an 
indecomposable tilting modules in $\mathcal{O}_0$ is given by 
Lusztig's $\mathbf{a}$-function (\cite{Lu,Lu2}). In \cite{Ma2}  this 
conjecture was proved in the case $\mathfrak{g}=\mathfrak{sl}_n$ (type $A$).
The proof consisted of two parts. In the first part it was shown that 
the projective dimension of an indecomposable tilting module
in $\mathcal{O}_0$ is an invariant of a two-sided cell (this part does
not depend on the type of $\mathfrak{g}$). The second part
was computational, computing the projective dimension for certain 
indecomposable tilting modules, however, the computation was based on 
the Koszul self-duality of $\mathcal{O}_0$ (\cite{So}) and computations
of graded filtrations of certain modules in  $\mathcal{O}_0$ in type
$A$ (\cite{Ir}). Computations in \cite{Ir} and in the subsequent
paper \cite{IS} covered also some special cases for other types and to
these case the arguments from \cite{Ma2} extend naturally. However, from
\cite{IS} it was also known that the arguments from \cite{Ir} and \cite{IS}
certainly do not extend to the general case. Hence, to connect 
tilting modules and Lusztig's $\mathbf{a}$-function in full generality
one had to come up with a completely different approach then the one
I proposed in \cite{Ma2}.

The main objective of the present paper is to prove the mentioned above
conjecture from \cite{Ma2} in full generality. We also prove a similar
conjecture from \cite{Ma2} about the projective dimension of indecomposable
injective modules. The proposed argument 
makes a surprising connection to another part of the paper
\cite{Ma2}. The category $\mathcal{O}_0$ is equivalent to the category
of modules over a finite-dimensional Koszul algebra (\cite{BGG,So}),
in particular, one can consider the corresponding category 
$\mathcal{O}_0^{\mathbb{Z}}$ of graded modules. In this situation an
important role is played by the category of the so-called linear 
complexes of tilting modules (\cite{Ma,MO2,MOS}). A part of \cite{Ma2}
is dedicated to showing that many structural modules from 
$\mathcal{O}_0$ (and from the parabolic subcategories of $\mathcal{O}_0$
in the sense of \cite{RC}) can be described using linear 
complexes of tilting modules. In the present paper we establish yet
another class of such modules, namely, the modules obtained from
simple modules using projective functors (\cite{BG}). In fact, we
even show that with respect to the Koszul self-duality of 
$\mathcal{O}_0$ this class of modules is Koszul self-dual
(other families of Koszul self-dual modules, for example 
shuffled Verma modules, can be found in \cite{Ma2}). After this we show
that certain numerical invariants of those linear complexes of tilting modules,
which represent the images of simple modules under projective functors,
are given in terms of Lusztig's $\mathbf{a}$-function. The conjecture
from \cite{Ma2} follows then using computations in the derived category.

By \cite{BG}, the action of projective functors on $\mathcal{O}_0$
can be considered as a categorification of the right regular 
representation of the Weyl group of $\mathfrak{g}$ (or the corresponding
Hecke algebra in the case of the category $\mathcal{O}_0^{\mathbb{Z}}$,
see \cite{MS4}). The images of simple modules under certain projective 
functors appear in \cite{MS3} for the case $\mathfrak{g}=\mathfrak{sl}_n$ 
as categorical interpretations of elements in 
a certain basis, in which the regular
representation of the symmetric group decomposes into a direct sum of 
cell modules (which are irreducible in type $A$). In the present paper
images of simple modules under projective functors appear naturally
in the general case. So, we extend the above result to the general case, 
generalizing \cite{MS3}, which, in particular, establishes certain
interesting facts about these modules. For example, we show that 
the images of simple modules under projective functors, which appear 
in our picture, have simple head and simple socle. We also confirm
\cite[Conjecture~2]{KM} about the structure of the dominant
projective module in certain subcategories of $\mathcal{O}_0$.

The paper is organized as follows: In Section~\ref{s1} we describe
the setup and introduce all necessary notation. In Section~\ref{s2}
we study the images of simple modules under projective functors,
in particular, we establish their Koszul self-duality and extend
the results from \cite{MS3} to the general case. In Section~\ref{s3}
we prove the conjecture from \cite{Ma2} about the connection 
between the projective dimension of indecomposable tilting modules 
in $\mathcal{O}_0$ and Lusztig's $\mathbf{a}$-function. In Section~\ref{s4}
we prove the conjecture from \cite{Ma2} about the connection 
between the projective dimension of indecomposable injective modules 
in $\mathcal{O}_0$ and Lusztig's $\mathbf{a}$-function.

\vspace{2mm}

\noindent
{\bf Acknowledgments.} The research was partially supported by The 
Royal Swedish Academy of  Sciences and  The Swedish Research Council. 
I would like to thank Henning Haahr Andersen and Catharina Stroppel  
for stimulating discussions.

\section{Preliminaries}\label{s1}

\subsection{Category $\mathcal{O}$}\label{s1.1}

I refer the reader to \cite{BGG,Hu,Ma2,MS4} for more details
on the category $\mathcal{O}$ and the notation I will use. 
Let $W$ denote the Weyl group of $\mathfrak{g}$ and $w_0$ be
the longest element of $W$. Denote by $\mathtt{A}$
the basic finite-dimensional associative algebra, whose category 
$\mathtt{A}\text{-}\mathrm{mod}$ of (left) finite-dimensional modules is 
equivalent to $\mathcal{O}_0$ (\cite{BGG}). We fix the Koszul grading
$\displaystyle\mathtt{A}=\bigoplus_{i\geq 0} \mathtt{A}_i$ 
on $\mathtt{A}$ (\cite{So,BGS})
and denote by $\mathtt{A}\text{-}\mathrm{gmod}$ the category 
of finite-dimensional graded $\mathtt{A}$-modules with degree zero
morphisms (this category is equivalent to the category 
$\mathcal{O}_0^{\mathbb{Z}}$ mentioned above). For $k\in\mathbb{Z}$ 
we denote by $\langle k\rangle$ the autoequivalence of
$\mathtt{A}\text{-}\mathrm{gmod}$, which shifts (decreases) the
degree of homogeneous components of a graded module by $k$.

Simple modules in $\mathcal{O}_0$ are indexed by elements of $W$
in the natural way. For $w\in W$ we denote by 
$L(w)$ the simple graded $\mathtt{A}$-module corresponding to $w$, 
concentrated in degree zero (here $L(e)$ corresponds to the trivial
module in $\mathcal{O}_0$ and $L(w_0)$ corresponds to the simple
Verma module in $\mathcal{O}_0$). We denote by $P(w)$ the projective 
cover of  $L(w)$ in $\mathtt{A}\text{-}\mathrm{gmod}$; by $I(w)$ the 
injective envelop of  $L(w)$ in $\mathtt{A}\text{-}\mathrm{gmod}$; 
by $\Delta(w)$ the standard quotient of $P(w)$ (a Verma module) and  
by $\nabla(w)$ the costandard submodule of $I(w)$ (a dual Verma module). 
Finally, we
denote by $T(w)$ the indecomposable tilting module, corresponding
to $w$, whose grading is uniquely determined by the condition that 
$\Delta(w)$ is a submodule of $T(w)$.

For $w\in W$ we denote by $\theta_w$ the graded version (\cite{St}) 
of the indecomposable projective endofunctor of
$\mathtt{A}\text{-}\mathrm{gmod}$ corresponding to $w$ (\cite{BG}). 
The functor $\theta_w$ is normalized by the
condition $\theta_w P(e)\cong P(w)$ (as graded modules).
The functor $\theta_w$ is both left and right adjoint to 
$\theta_{w^{-1}}$.

Denote by $D^b(\mathtt{A})$ the bounded derived category of
$\mathtt{A}\text{-}\mathrm{gmod}$ and by 
$\llbracket k\rrbracket$, $k\in\mathbb{Z}$,
the autoequivalence of $D^b(\mathtt{A})$, which shifts complexes by
$k$ positions to the left. We denote by $\mathfrak{LT}$ the full
subcategory of $D^b(\mathtt{A})$, which consists of all complexes
\begin{displaymath}
\mathcal{X}^{\bullet}:\quad\quad\quad
\dots \rightarrow \mathcal{X}^{-1}\rightarrow 
\mathcal{X}^{0}\rightarrow \mathcal{X}^{1}
\rightarrow \dots
\end{displaymath}
such that for every $i\in\mathbb{Z}$ the module $\mathcal{X}^i$ is isomorphic
to a direct sum of modules of the form $T(w)\langle i\rangle$, $w\in W$.
The category $\mathfrak{LT}$ is abelian with enough projectives,
moreover, there is an equivalence of categories  $\Phi:\mathfrak{LT}\to
\mathtt{A}\text{-}\mathrm{gmod}$  (\cite{MO2,Ma}). From 
\cite[Theorem~3.3(1)]{Ma} we have that $\Phi$ sends the indecomposable 
tilting module $T(w)$, $w\in W$, considered as a linear complex 
concentrated in position zero, to the simple module $L(w_0w^{-1}w_0)$.

We denote by $\star:D^b(\mathtt{A})\to D^b(\mathtt{A})$ the usual 
contravariant autoequivalence preserving isoclasses of simple modules 
concentrated in degree zero (duality). All projective functors 
commute with $\star$. For $M\in 
\mathtt{A}\text{-}\mathrm{gmod}$, $M\neq 0$, we set
\begin{displaymath}
\min(M)=\min\{i\in\mathbb{Z}:M_i\neq 0\},\quad\quad 
\max(M)=\max\{i\in\mathbb{Z}:M_i\neq 0\}.
\end{displaymath}

For $w\in W$ we denote by $\mathtt{T}_w:\mathtt{A}\text{-}\mathrm{gmod}
\to \mathtt{A}\text{-}\mathrm{gmod}$ the corresponding Arkhipov's
twisting functor (see \cite{AS,KhMa} for the ungraded version
and \cite[Appendix]{MO2} for the graded version).

\subsection{Kazhdan-Lusztig combinatorics}\label{s1.2}

Here I refer the reader to \cite[Section~3]{MS4}, \cite{So2} and 
\cite{BjBr} for details. Let $S$ be the set of simple reflections in $W$ and 
$l:W\rightarrow\mathbb{Z}$ be the length function with respect to $S$. 
Denote by $\mathbb{H}$ the Hecke algebra of $W$, which is a free 
$\mathbb{Z}[v,v^{-1}]$-module with basis $\{H_w:w\in W\}$ and multiplication
given by
\begin{displaymath}
H_xH_y=H_{xy} \text{ if } l(x)+l(y)=l(xy);
\text{ and } H_s^2=H_e+(v^{-1}-v)H_s, s\in S.
\end{displaymath}
Let $\{\underline{H}_w:w\in W\}$ and $\{\underline{\hat{H}}_w:w\in W\}$ 
denote the Kazhdan-Lusztig and the dual Kazhdan-Lusztig bases of
$\mathbb{H}$, respectively. 

Consider the Grothendieck group $[\mathtt{A}\text{-}\mathrm{gmod}]$
of $\mathtt{A}\text{-}\mathrm{gmod}$ and for $M\in 
\mathtt{A}\text{-}\mathrm{gmod}$ denote by $[M]$ the image of $M$
in $[\mathtt{A}\text{-}\mathrm{gmod}]$. Then the assignment
$[\Delta(w)\langle i\rangle]\mapsto v^{-i} H_w$ gives rise to an isomorphism
between $[\mathtt{A}\text{-}\mathrm{gmod}]$ and $\mathbb{H}$. In what follows
we will often identify $[\mathtt{A}\text{-}\mathrm{gmod}]$ and $\mathbb{H}$
via this isomorphism. For all $M\in \mathtt{A}\text{-}\mathrm{gmod}$
we have $[M\langle i\rangle]=v^{-i} [M]$. We also have
$[P(w)]=\underline{H}_w$ and $[L(w)]=\underline{\hat{H}}_w$ for all
$w\in W$. Furthermore, for any $w\in W$ and $M\in 
\mathtt{A}\text{-}\mathrm{gmod}$ we have 
$[\theta_w M]=[M]\underline{H}_w$. All the above extends to 
$D^b(\mathtt{A})$ in the obvious way.

Further, we denote by $\leq_L$, $\leq_R$ and $\leq_{LR}$ the left, the right
and the two-sided orders on $W$, respectively (to make things coherent
with \cite{MS4} our convention is that $e$ is the minimal element and
$w_0$ is the maximal element). The equivalence classes with
respect to these orders are called left-, right- and two-sided cells of $W$,
respectively. The corresponding equivalence relations will be denoted by
$\sim_{L}$, $\sim_{R}$ and  $\sim_{LR}$, respectively. 
Let $\mathbf{a}:W\to \mathbb{Z}$ be Lusztig's 
$\mathbf{a}$-function on $W$ (\cite{Lu,Lu2}). This function respects the
two-sided order, in particular, it is constant on
two-sided cells. On the (unique) distinguished (Duflo) involution $w$ from 
a given left cell we have $\mathbf{a}(w)=l(w)-2\delta(w)$, 
where $\delta(w)$ is the degree of the Kazhdan-Lusztig polynomial $P_{e,w}$.
Since $\mathbf{a}$ is constant on two-sided cells, we sometimes will
write $\mathbf{a}(X)$, where $X$ is cell (left, right, or two-sided),
meaning $\mathbf{a}(x)$, $x\in X$.

Any right cell $\mathbf{R}$ comes equipped with a natural dual 
Kazhdan-Lusztig basis. Multiplication with elements from the
Kazhdan-Lusztig basis respects the right preorder and produces elements,
which are linear combinations of elements from the dual 
Kazhdan-Lusztig basis of $\mathbf{R}$ or smaller right cells. Taking the
quotient defines on the linear span of elements from the dual
Kazhdan-Lusztig basis of $\mathbf{R}$ the structure of
an $\mathbb{H}$-module, called the (right) cell module.

\subsection{Subcategories of $\mathcal{O}$ associated with right 
cells}\label{s1.3}

For a fixed right cell $\mathbf{R}$ of $W$ let $\hat{\mathbf{R}}$ denote 
the set of  all elements $x\in W$ such that $x\leq_R w$ for some 
$w\in \mathbf{R}$. Let $\{e_w:w\in W\}$ be a complete set of 
primitive idempotents of $\mathtt{A}$ corresponding to our indexing
of simple modules. Following \cite{MS4} set 
$\displaystyle e(\hat{\mathbf{R}})=\sum_{w\not\in \hat{\mathbf{R}}}e_w$ and
consider the quotient
\begin{displaymath}
\mathtt{A}^{\hat{\mathbf{R}}}=
\mathtt{A}/\mathtt{A}e(\hat{\mathbf{R}})\mathtt{A}.
\end{displaymath}
Then $\mathtt{A}^{\hat{\mathbf{R}}}\text{-}\mathrm{gmod}$ is a 
subcategory of $\mathtt{A}\text{-}\mathrm{gmod}$ in the natural
way (this is the Serre subcategory of $\mathtt{A}\text{-}\mathrm{gmod}$
generated by the simple modules of the form $L(w)\langle k\rangle$,
$w\in \hat{\mathbf{R}}$, $k\in\mathbb{Z}$). The category
$\mathtt{A}^{\hat{\mathbf{R}}}\text{-}\mathrm{gmod}$ is stable
under all projective functors $\theta_w$, $w\in W$. 
See \cite{MS4} for details.
The structural modules in $\mathtt{A}^{\hat{\mathbf{R}}}\text{-}\mathrm{gmod}$
will be normally denoted similarly to the corresponding modules from
$\mathtt{A}\text{-}\mathrm{gmod}$ but with an extra index
$\hat{\mathbf{R}}$. 

We denote by $\mathrm{Z}^{\hat{\mathbf{R}}}:
\mathtt{A}\text{-}\mathrm{gmod}\to 
\mathtt{A}^{\hat{\mathbf{R}}}\text{-}\mathrm{gmod}$ the left adjoint
of the natural inclusion of 
$\mathtt{A}^{\hat{\mathbf{R}}}\text{-}\mathrm{gmod}$ into 
$\mathtt{A}\text{-}\mathrm{gmod}$. The functor $\mathrm{Z}^{\hat{\mathbf{R}}}$
is just the functor of taking the maximal possible quotient which belongs
to $\mathtt{A}^{\hat{\mathbf{R}}}\text{-}\mathrm{gmod}$. This functor
commutes with all projective functors \cite[Lemma~19]{MS4} and satisfies
\begin{displaymath}
\mathrm{Z}^{\hat{\mathbf{R}}}L(w)=
\begin{cases}
L^{\hat{\mathbf{R}}}(w), & w\in  \hat{\mathbf{R}};\\
0, & \text{otherwise};
\end{cases}
 \quad
\mathrm{Z}^{\hat{\mathbf{R}}}P(w)=
\begin{cases}
P^{\hat{\mathbf{R}}}(w), & w\in  \hat{\mathbf{R}};\\
0, & \text{otherwise}.
\end{cases}
\end{displaymath}

In the case when $\mathbf{R}$ contains an element of the form 
$xw_0$, where $x$ is the longest element in some parabolic subgroup of
$W$, the category $\mathtt{A}^{\hat{\mathbf{R}}}\text{-}\mathrm{gmod}$
is equivalent to the graded version of the parabolic category 
$\mathcal{O}$ in the sense of \cite{RC}. In this case 
$\mathrm{Z}^{\hat{\mathbf{R}}}$ is the corresponding Zuckerman functor
(\cite{MS25}).

\section{Images of simple modules under projective functors}\label{s2}

In this section we will study modules $M(x,y)=\theta_xL(y)$,
$x,y\in W$. On the level of the Hecke algebra we have
$[M(x,y)]=\underline{\hat{H}}_y\underline{H}_x\in\mathbb{H}$.
The latter elements of $\mathbb{H}$ play an important role in the 
combinatorics of $\mathbb{H}$, see \cite{Lu,Lu2,Mat,Ne}.

\subsection{Graded lengths of $M(x,y)$}\label{s2.1}

We start with the following result which describes nonzero homogeneous
components of the module $M(x,y)$.

\begin{proposition}\label{prop1}
For all $x,y\in W$ we have:
\begin{enumerate}[(a)]
\item\label{prop1-1} $M(x,y)^{\star}\cong M(x,y)$;
\item\label{prop1-2} $M(x,y)\neq 0$ if and only if $x\leq_R y^{-1}$.
\item\label{prop1-3} $\max(M(x,y))=-\min(M(x,y))=\mathbf{a}(y)$
whenever $x\sim_{R}y^{-1}$, and $\max(M(x,y))=-\min(M(x,y))<\mathbf{a}(y)$ 
whenever $x<_{R}y^{-1}$.
\end{enumerate}
\end{proposition}

\begin{proof}
The module $L(y)$ is simple and hence selfdual (i.e. satisfies
$L(y)^{\star}\cong L(y)$). Now the claim \eqref{prop1-1} follows 
from the fact that $\star$ and $\theta_x$ commute. This implies
the equality $\max(M(x,y))=-\min(M(x,y))$ in \eqref{prop1-3}.

The rest of the statements is purely combinatorial and follows
from well-known properties of (dual) Kazhdan-Lusztig bases in the
Hecke algebra. To start with, the claim \eqref{prop1-2} follows 
from \cite[(1.4)]{Mat}.

That $\max(M(x,y))=\mathbf{a}(y)$ in the case $x\sim_{R}y^{-1}$ 
follows from the definition of $\mathbf{a}$ and the explicit formula 
for $\underline{\hat{H}}_y\underline{H}_x$, see
for example \cite[Lemma~2.1]{Mat}. Similarly,
that $\max(M(x,y))<\mathbf{a}(y)$ whenever $M(x,y)\neq 0$ 
and $x<_{R}y^{-1}$ follows from the definition of $\mathbf{a}$,
\cite[Lemma~2.1]{Mat} and the fact that $\mathbf{a}$ respects the
right order (\cite[Corollary~6.3]{Lu}). This  completes the proof.
\end{proof}

\subsection{The dominant projective module in 
$\mathtt{A}^{\hat{\mathbf{R}}}\text{-}\mathrm{gmod}$}\label{s2.2}

Our next goal is to show that modules $M(x,y)$ (for certain 
choices of $x$ and $y$) are projective-injective modules in the 
category $\mathtt{A}^{\hat{\mathbf{R}}}\text{-}\mathrm{gmod}$. 
To prove this we first have to describe the dominant projective 
module $P^{\hat{\mathbf{R}}}(e)$ in 
$\mathtt{A}^{\hat{\mathbf{R}}}\text{-}\mathrm{gmod}$.
In what follows we assume that $\mathbf{R}$ is a fixed right cell of $W$.

\begin{proposition}\label{prop2}
Let $x\in \mathbf{R}$. There is a unique (up to scalar) nonzero 
homomorphism from $\Delta(e)$ to $\theta_{x^{-1}}L(x)$ and the
module $P^{\hat{\mathbf{R}}}(e)$ coincides with the image of
this homomorphism.
\end{proposition}

\begin{proof}
By adjunction we have 
\begin{displaymath}
\begin{array}{rcl}
\mathrm{Hom}_{\mathtt{A}}(\Delta(e),\theta_{x^{-1}}L(x))
&=&\mathrm{Hom}_{\mathtt{A}}(\theta_{x}\Delta(e),L(x))\\
&=&\mathrm{Hom}_{\mathtt{A}}(P(x),L(x))\\&=&\mathbb{C},
\end{array}
\end{displaymath}
which proves the first part of the claim. Let $D$ denote the image 
of $\Delta(e)$ in $\theta_{x^{-1}}L(x)$. By \cite[Proposition~5.1]{Ka},
the module $D$ does not depend on the choice of $x$. Choose $x$ the
maximal possible (with respect to the Bruhat order).  Then the only
simple subquotient from 
$\mathtt{A}^{\hat{\mathbf{R}}}\text{-}\mathrm{gmod}$
in the module $T(x)$ is $L(x)$, so we have the complex
\begin{displaymath}
0\to M_{-1} \to T(x) \to M_1 \to 0, 
\end{displaymath}
which has only one nonzero homology, namely $L(x)$ in the zero position.
All simple subquotients of both $M_{-1}$ and $M_1$ do not belong to 
$\mathtt{A}^{\hat{\mathbf{R}}}\text{-}\mathrm{gmod}$. Applying the
exact functor $\theta_{x^{-1}}$ we get the complex
\begin{equation}\label{eq1}
0\to \theta_{x^{-1}} M_{-1} \to \theta_{x^{-1}} T(x) \to 
\theta_{x^{-1}} M_1 \to 0, 
\end{equation}
which has only one nonzero homology, namely $\theta_{x^{-1}} L(x)$ 
in the zero position. Note that this homology belongs to
$\mathtt{A}^{\hat{\mathbf{R}}}\text{-}\mathrm{gmod}$ as
$L(x)$ does and $\theta_{x^{-1}}$ preserves this category.

Since $T(x)$ has a unique occurrence of $L(x)$ (counting all shifts in
grading as well), by adjunction there is a unique (up to scalar)
nonzero morphism from
$\Delta(e)$ to $\theta_{x^{-1}} T(x)$, which is a tilting module. 
Since $\Delta(e)$ is the dominant Verma module, we get that 
$\theta_{x^{-1}} T(x)$ must contain $T(e)$ as a direct summand
(with multiplicity one, counting with all shifts) and 
the above homomorphism from $\Delta(e)$ to $\theta_{x^{-1}} T(x)$ 
is the natural injection from $\Delta(e)$ into this direct summand.

Let $y\in W$. Then, by adjunction, for every $i\in\mathbb{Z}$
we have
\begin{displaymath}
\mathrm{Hom}_{\mathtt{A}}(\theta_{x^{-1}} M_{-1},L(y)\langle i\rangle)= 
\mathrm{Hom}_{\mathtt{A}}(M_{-1},\theta_{x} L(y)\langle i\rangle).
\end{displaymath}
Since $\theta_{x}$ preserves 
$\mathtt{A}^{\hat{\mathbf{R}}}\text{-}\mathrm{gmod}$, the space on the
right hand side can be nonzero only if $y\not\in\mathbf{\hat{R}}$.  
It  thus follows
that every simple module occurring in the head of $M_{-1}$ does not
belong to $\mathtt{A}^{\hat{\mathbf{R}}}\text{-}\mathrm{gmod}$.
Similarly, all simple modules occurring in the socle of 
$\theta_{x^{-1}} M_1$ do not  belong to 
$\mathtt{A}^{\hat{\mathbf{R}}}\text{-}\mathrm{gmod}$.

From the above we know that the module $\theta_{x^{-1}} T(x)$ has a 
unique simple  subquotient isomorphic to $L(e)$ (counting all 
shifts of grading),  which, moreover, appears in the homology of
the sequence \eqref{eq1}. This means that the monomorphism from 
$\Delta(e)$ to $\theta_{x^{-1}} T(x)$ induces a homomorphism from
$P^{\hat{\mathbf{R}}}(e)=\mathrm{Z}^{\hat{\mathbf{R}}}\Delta(e)$
to the homology $\theta_{x^{-1}} L(x)$. From the two previous paragraphs
it follows that this homomorphism is injective. This completes the proof.
\end{proof}

\begin{corollary}\label{cor3}
Let $d$ be the Duflo involution in $\mathbf{R}$. The module
$P^{\hat{\mathbf{R}}}(e)$ has simple socle
$L(d)\langle-\mathbf{a}(d)\rangle$, and all other composition
subquotients of the form $L(x)\langle -i\rangle$, where
$x<_{LR}d$ and $0\leq i< \mathbf{a}(d)$.
\end{corollary}

\begin{proof}
This follows from  Proposition~\ref{prop2} and \cite[Proposition~5.1]{Ka}.
\end{proof}

\begin{remark}\label{rem4}
{\rm
Proposition~\ref{prop2} proves \cite[Conjecture~2]{KM}.
}
\end{remark}

\begin{remark}\label{rem5}
{\rm
Using \cite[Proposition~5.1]{Ka} one can relate the module
$P^{\hat{\mathbf{R}}}(e)$ to a primitive quotient of 
the universal enveloping algebra $U(\mathfrak{g})$.
}
\end{remark}

\subsection{Projective-injective modules in 
$\mathtt{A}^{\hat{\mathbf{R}}}\text{-}\mathrm{gmod}$}\label{s2.3}

Using the results from the previous subsection we obtain the
following:

\begin{theorem}\label{thm6}
Let $d\in \mathbf{R}$ be the Duflo involution. Then the modules
$M(x,d)$, $x\in\mathbf{R}$, are exactly the indecomposable
projective-injective modules in 
$\mathtt{A}^{\hat{\mathbf{R}}}\text{-}\mathrm{gmod}$
(up to shift).
\end{theorem}

\begin{proof}
Let $x\in\mathbf{R}$. Applying $\theta_x$ to the short exact sequence
\begin{displaymath}
0\to L(d)\langle-\mathbf{a}(d)\rangle\to 
 P^{\hat{\mathbf{R}}}(e)\to\mathrm{Coker}\to 0,
\end{displaymath}
given by Corollary~\ref{cor3}, we obtain the short exact sequence
\begin{displaymath}
0\to M(x,d)\langle-\mathbf{a}(d)\rangle\to 
P^{\hat{\mathbf{R}}}(x)\to\theta_x\mathrm{Coker}\to 0.
\end{displaymath}
From Corollary~\ref{cor3} and Proposition~\ref{prop1}\eqref{prop1-2}
we have $\theta_x\mathrm{Coker}=0$ and hence 
$M(x,d)\langle-\mathbf{a}(d)\rangle\cong 
P^{\hat{\mathbf{R}}}(x)$. This shows that the module $M(x,d)$ is 
projective in $\mathtt{A}^{\hat{\mathbf{R}}}\text{-}\mathrm{gmod}$.
From Proposition~\ref{prop1}\eqref{prop1-1} we have that 
$M(x,d)$ is self-dual, hence it is injective in 
$\mathtt{A}^{\hat{\mathbf{R}}}\text{-}\mathrm{gmod}$ as well.

On the other hand, for $x\in \mathbf{\hat{R}}\setminus\mathbf{R}$
we have that $P^{\hat{\mathbf{R}}}(x)=\theta_xP^{\hat{\mathbf{R}}}(e)$
has simple top $L(x)$. At the same time, as $P^{\hat{\mathbf{R}}}(e)$
has simple socle $L(d)$ (up to shift), using adjunction and arguments 
similar to those used in the proof of Proposition~\ref{prop2}, one 
shows that every simple submodule of $P^{\hat{\mathbf{R}}}(x)$ must have 
the  form $L(y)$, $y\in \mathbf{R}$ (up to shift). Therefore the injective
cover of $P^{\hat{\mathbf{R}}}(x)$ does not coincide with 
$P^{\hat{\mathbf{R}}}(x)$ by the previous paragraph.
Hence the module $P^{\hat{\mathbf{R}}}(x)$ is not 
injective. This completes the proof.
\end{proof}

The following result generalizes some results of \cite{Ir0}:

\begin{corollary}\label{cor7501}
The Loewy length of every projective-injective module in 
$\mathtt{A}^{\hat{\mathbf{R}}}\text{-}\mathrm{gmod}$ equals
$2\mathbf{a}(\mathbf{R})+1$.
\end{corollary}

\begin{proof}
Let $X$ be an indecomposable projective-injective module
in the category $\mathtt{A}^{\hat{\mathbf{R}}}\text{-}\mathrm{gmod}$. 
From Theorem~\ref{thm6} and Proposition~\ref{prop1} we obtain
that $2\mathbf{a}(\mathbf{R})+1$ is the graded length of this module
(the number of nonzero homogeneous components). As the algebra
$\mathtt{A}$ is Koszul, it is positively graded and generated in 
degrees zero and one. Hence the quotient algebra
$\mathtt{A}^{\hat{\mathbf{R}}}$ is positively graded and generated in 
degrees zero and one as well. Since $X$ has both simple socle and simple
head (by Theorem~\ref{thm6}), from \cite[Proposition~2.4.1]{BGS} we thus
obtain that the graded filtration of $X$ is a Loewy filtration.
The claim follows.
\end{proof}

\begin{corollary}\label{cor7}
The injective envelope of $P^{\hat{\mathbf{R}}}(e)$ is
$P^{\hat{\mathbf{R}}}(d)\langle \mathbf{a}(d)\rangle$.
\end{corollary}

\begin{proof}
This follows from Theorem~\ref{thm6} and Corollaries~\ref{cor3}
and \ref{cor7501}. 
\end{proof}

\begin{corollary}\label{cor75}
Let $d\in \mathbf{R}$ be the Duflo involution. Then all modules
$M(x,d)$, $x\in\mathbf{R}$, have their simple socles and simple heads
in degrees $\mathbf{a}(x)$ and $-\mathbf{a}(x)$, respectively.
\end{corollary}

\begin{proof}
This follows from the positivity of the grading on
$\mathtt{A}$, Proposition~\ref{prop1} and Theorem~\ref{thm6}.
\end{proof}

\begin{remark}\label{rem755}
{\rm  
Projective-injective modules play important role in the structure
and properties of the category $\mathcal{O}$ and related categories,
see \cite{Ir0,MS,MS4} and references therein. 
}
\end{remark}

\subsection{Application to Kostant's problem}\label{s2.4}

Results from the previous subsections can be applied to one classical
problem in Lie theory, called Kostant's problem (\cite{Jo}). 
If $M,N$ are two $\mathfrak{g}$-modules, then 
$\mathrm{Hom}_{\mathbb{C}}(M,N)$ is a $\mathfrak{g}$-bimodule in the
natural way. Denote by $\mathcal{L}(M,N)$ the subbimodule of 
$\mathrm{Hom}_{\mathbb{C}}(M,N)$, which consists of all elements,
the adjoint action of $\mathfrak{g}$ on which is locally finite.
Then for any $\mathfrak{g}$-module $M$ the universal enveloping
algebra $U(\mathfrak{g})$ maps naturally to $\mathcal{L}(M,M)$ 
inducing an injection
\begin{equation}\label{eq2}
U(\mathfrak{g})/\mathrm{Ann}_{U(\mathfrak{g})}(M)\hookrightarrow
\mathcal{L}(M,M).
\end{equation}
Kostant's problem for $M$ is to determine whether the latter map
is surjective. The problem is very hard and the answer is not even
known for the modules $L(w)$, $w\in W$, in the general case, although
many special cases are settled (see \cite{Jo,Ma3,MS3,MS4,Ka,KM} are references
therein). Taking into account the results of the previous subsections,
the main result of \cite{KM} can be formulated as follows:

\begin{theorem}[\cite{KM}]\label{thm8}
Let $d\in \mathbf{R}$ be the Duflo involution. Then 
Kos\-tant's problem for $L(d)$ has a positive answer (i.e. the map from
\eqref{eq2} is surjective) if and only if the only simple modules
occurring in the socle of the cokernel of the natural injection
$P^{\hat{\mathbf{R}}}(e)\hookrightarrow
P^{\hat{\mathbf{R}}}(d)\langle \mathbf{a}(d)\rangle$
(given by Corollary~\ref{cor7})
are (up to shit) the modules of the form $L(x)$, $x\in \mathbf{R}$. 
\end{theorem}

After Theorem~\ref{thm6} the above can be reformulated in terms of the
so-called double-centralizer property (see \cite{So,KSX,MS5}). Let $A$
be a finite-dimensional algebra and $X$ be a left $A$-module. Then
$A$ has the double centralizer property with respect to $X$ if there is
an exact sequence
\begin{displaymath}
0\to {}_A A\to X_1 \to X_2,
\end{displaymath}
where both $X_1$ and $X_2$ are isomorphic to finite direct sums of 
some  direct summands of $X$. Double centralizer properties play important
role in the representation theory (see \cite{So,KSX,MS5}). 
In our case we have:

\begin{corollary}\label{cor9}
Let $d\in \mathbf{R}$ be the Duflo involution. Then 
Kos\-tant's problem for $L(d)$ has a positive answer if and only if
$\mathtt{A}^{\hat{\mathbf{R}}}$ has the double centralizer property
with respect to the direct sum of all
indecomposable projective-injective modules.
\end{corollary}

\begin{proof}
Let $X$ denote the cokernel of the natural inclusion
$P^{\hat{\mathbf{R}}}(e)\hookrightarrow
P^{\hat{\mathbf{R}}}(d)\langle \mathbf{a}(d)\rangle$ given by
Corollary~\ref{cor7}. Assume that Kostant's problem has a positive
answer for $L(d)$. Then, because of Theorem~\ref{thm8} and
Theorem~\ref{thm6}, the injective envelope $I$ of $X$ is also projective
and hence we have an exact sequence
\begin{displaymath}
0\to P^{\hat{\mathbf{R}}}(e) \to  
P^{\hat{\mathbf{R}}}(d)\langle \mathbf{a}(d)\rangle \to I,
\end{displaymath}
where the two last terms are both projective and injective. Applying
$\theta_x$, $x\in\mathbf{\hat{R}}$, gives the exact sequence
\begin{displaymath}
0\to P^{\hat{\mathbf{R}}}(x) \to  
\theta_x P^{\hat{\mathbf{R}}}(d)\langle -\mathbf{a}(d)\rangle \to 
\theta_x I,
\end{displaymath}
where again the two last terms are both projective and injective
since $\theta_x$ preserves both projectivity and injectivity. This implies
that $\mathtt{A}^{\hat{\mathbf{R}}}$ has the double centralizer property
with respect to the direct sum of all
indecomposable projective-injective modules.

On the other hand, if Kostant's problem has a negative answer for $L(d)$,
then, because of Theorem~\ref{thm8} and
Theorem~\ref{thm6}, the injective envelope $I$ of $X$ is not projective.
Therefore $\mathtt{A}^{\hat{\mathbf{R}}}$ does not have the double 
centralizer property with respect to the direct sum of all
indecomposable projective-injective modules.
This completes the proof.
\end{proof}

\begin{remark}\label{rem10}
{\rm 
In the case $\mathfrak{g}=\mathfrak{sl}_n$ (type $A$) Corollary~\ref{cor9}
controlls the answer to Kostant's problem for all $L(w)$, $w\in W$,
as this answer is known to be a left cell invariant  (\cite{MS4}).
}
\end{remark}

\subsection{Regular $\mathbb{H}$-module as a sum of cell modules}\label{s2.5}

In this subsection we extend the results of \cite{MS3} and \cite{KMS}
to the general
case. For $w\in W$ let $d_w$ denote the Duflo involution in the right
cell of $W$. For $x\in W$ denote by $[\theta_x]$ the linear operator
on $[\mathtt{A}\text{-}\mathrm{gmod}]$, induced by the exact functor
$\theta_x$. We have the following categorification result:

\begin{theorem}\label{thm11}
\begin{enumerate}[(a)]
\item\label{thm11-1} 
The action of $[\theta_x]$, $x\in W$, on 
$[\mathtt{A}\text{-}\mathrm{gmod}]$ gives a right regular representation
of $\mathbb{H}$ in the Kazhdan-Lusztig basis.
\item\label{thm11-2}
The classes $[M(w,d_w)]$, $w\in W$, form a basis
of the complex vector space 
$\mathbb{C}\otimes_{\mathbb{Z}}[\mathtt{A}\text{-}\mathrm{gmod}]/(v-1)$,
on which the action of $[\theta_x]$, $x\in W$, gives a right regular 
representation of $W$.
\item\label{thm11-3} In the basis from \eqref{thm11-2} the 
right regular  representation of $W$ decomposes into a direct sum
of (right) cell modules.
\end{enumerate}
\end{theorem}

\begin{proof}
Using Theorem~\ref{thm6} and Corollary~\ref{cor7} the proof is
similar to that of the main result from \cite{MS3}.
\end{proof}

\begin{remark}\label{rem12}
{\rm  
Theorem~\ref{thm11} gives a category theoretical interpretation 
of a basis in Lusztig's asymptotic Hecke algebra (\cite{Lu2,Ne}).
}
\end{remark}

\subsection{Koszul self-duality}\label{s2.6}

In this subsection we prove the following crucial result which 
establishes Koszul self-duality for modules $M(x,y)$:

\begin{theorem}\label{thm14}
Let $x,y\in W$.
\begin{enumerate}[(a)]
\item\label{thm14-1} There is $\mathcal{M}(x,y)^{\bullet}\in
\mathfrak{LT}$ that has a unique nonzero homology which is in position
zero and is isomorphic to $M(x,y)$.
\item\label{thm14-2} $\Phi \mathcal{M}(x,y)^{\bullet}\cong
M(y^{-1}w_0,w_0x^{-1})$.
\end{enumerate}
\end{theorem}

\begin{proof}
We prove both statements by a descending induction on the length of $y$.
Let $w\in W$. Consider $T(w)$ as a linear complex concentrated in position
zero (this is a simple object of $\mathfrak{LT}$). 
We have $\Phi T(w)\cong L(w_0w^{-1}w_0)$. On the other hand, we
also have
\begin{displaymath}
T(w)\cong \theta_{w_0w}T(w_0)\cong \theta_{w_0w}\Delta(w_0) 
\cong \theta_{w_0w}L(w_0)=M(w_0w,w_0). 
\end{displaymath}
Hence our complex consisting of $T(w)$ is exactly 
$\mathcal{M}(w_0w,w_0)^{\bullet}$ and we have 
\begin{displaymath}
\Phi \mathcal{M}(w_0w,w_0)^{\bullet}\cong M(e,w_0w^{-1}w_0),
\end{displaymath}
which agrees with \eqref{thm14-2}. This proves the basis of the induction.

Assume now the the statement is true for all $y\in W$ such that $l(y)>k$,
where $0\leq k< l(w_0)$, and let $y\in W$ be such that $l(y)=k$.
Let $s\in S$ be such that $l(y^{-1}w_0s)<l(y^{-1}w_0)$ and
$\overline{y}=w_0(y^{-1}w_0s)^{-1}$. Then $l(\overline{y})>k$ and hence
the claim of the theorem is true for all $M(x,\overline{y})$,
$x\in W$, by the inductive assumption.

For $x\in W$ take the linear complex $M(x,\overline{y})^{\bullet}$.
We have $\Phi  M(x,\overline{y})^{\bullet}\cong
M(y^{-1}w_0s,w_0x^{-1})$. As $l(y^{-1}w_0ss)>l(y^{-1}w_0s)$ by our choice
of $s$, applying $\theta_s$  to $M(y^{-1}w_0s,w_0x^{-1})$, using
\cite[(1)]{Ma2}, and going back to $\mathfrak{LT}$ via $\Phi^{-1}$,
we obtain a direct sum of linear complexes, where one direct summand
will be $\Phi^{-1} M(y^{-1}w_0,w_0x^{-1})$ and multiplicities of
other direct summand are determined by Kazhdan-Lusztig's 
$\mu$-function (\cite{KL,BjBr}) as given by \cite[(1)]{Ma2}.  

The Koszul dual of $\theta_s$ is the derived Zuckerman functor 
(\cite{RH,MOS}). This functor was explicitly described in \cite{MS25}. 
It has only three components. The first one takes the maximal quotient 
with subquotients of the form $L(w)$, $l(w_0sw_0w)>l(w)$, (with 
the corresponding shifts in  grading) and shifts it one position 
to the right. This component is zero because of our choice of $s$. The
second component is dual to the first one. It takes the maximal submodule 
with subquotients of the form $L(w)$, $l(w_0sw_0w)>l(w)$, (with 
the corresponding shifts in  grading) and shifts it one position 
to the left. This component is zero by the dual reason and 
Proposition~\ref{prop1}\eqref{prop1-1}. The only component which
is left is the functor $\mathrm{Q}$ from \cite[Theorem~5]{MS25}, so the
homology of the complex 
$\Phi^{-1} \theta_s M(y^{-1}w_0s,w_0x^{-1})$ is isomorphic to
$\mathrm{Q}M(x,w_0sw_0y)$ (and the homology of 
$\Phi^{-1} M(y^{-1}w_0,w_0x^{-1})$ is a direct summand which, as we will
see, is easy to track).

Let us now compute the module $\mathrm{Q}M(x,w_0sw_0y)$.
From  \cite[Theorem~5]{MS25} and \cite{KhMa} it follows that the functor
$\mathrm{Q}$ commutes with projective functors, in particular, we have
\begin{displaymath}
\mathrm{Q}M(x,w_0sw_0y)\cong \mathrm{Q}\theta_x L(w_0sw_0y)\cong
\theta_x\mathrm{Q} L(w_0sw_0y).
\end{displaymath}
Using  \cite[Theorem~5]{MS25} and \cite[Theorem~6.3]{AS} we obtain that
the module $\mathrm{Q} L(w_0sw_0y)$ is a direct sum of $L(y)$ and some other
simple modules, whose multiplicities are again determined 
by Kazhdan-Lusztig's $\mu$-function.
Hence, comparing \cite[Theorem~6.3]{AS} and \cite[(1)]{Ma2} and using
the inductive assumption we see that  these ``other simple modules'' 
precisely correspond to the direct summands of the linear complex 
$\Phi^{-1} \theta_s M(y^{-1}w_0s,w_0x^{-1})$, different from  
$\Phi^{-1} M(y^{-1}w_0,w_0x^{-1})$. Therefore the homology of 
$\Phi^{-1} M(y^{-1}w_0,w_0x^{-1})$ is exactly $\theta_xL(y)=M(x,y)$.
This completes the proof.
\end{proof}

\section{Projective dimension of indecomposable tilting modules}\label{s3}

Now we are ready to prove the main result of this paper
(see \cite[Conjecture~15(a)]{Ma2}).

\begin{theorem}\label{thm15}
Let $w\in W$. Then the projective dimension of the module $T(w)$
equals $\mathbf{a}(w)$.
\end{theorem}

Theorem~\ref{thm15} follows from Lemmata~\ref{lem16} and \ref{lem17}
below. Both here and in the next section we use the technique for
computation of extensions using complexes of tilting modules,
developed in \cite{MO}.

\begin{lemma}\label{lem16}
The projective dimension of $T(w)$
is at most $\mathbf{a}(w)$.
\end{lemma}

\begin{proof}
Let $y\in W$. We start with the following computation
(here the notation $\mathcal{O}$ means that we consider ungraded
versions of all modules):
\begin{displaymath}
\begin{array}{rcl} 
\mathrm{Ext}_{\mathcal{O}}^i(T(w),L(y))&\cong &
\mathrm{Hom}_{\mathcal{D}^b(\mathcal{O})}(T(w),L(y)\llbracket i\rrbracket)\\
&\cong &
\mathrm{Hom}_{\mathcal{D}^b(\mathcal{O})}(\theta_{w_0w}T(w_0),L(y)
\llbracket i\rrbracket)\\
\text{(by adjunction)}&\cong &
\mathrm{Hom}_{\mathcal{D}^b(\mathcal{O})}(T(w_0),\theta_{(w_0w)^{-1}}L(y)
\llbracket i\rrbracket)\\
&\cong &
\mathrm{Hom}_{\mathcal{D}^b(\mathcal{O})}(T(w_0),M((w_0w)^{-1},y)
\llbracket i\rrbracket).\\
\end{array}
\end{displaymath}
By Theorem~\ref{thm14}\eqref{thm14-1},
the module $M((w_0w)^{-1},y)$ can be represented in the derived category 
by the complex $M((w_0w)^{-1},y)^{\bullet}$.
Applying Theorem~\ref{thm14}\eqref{thm14-2}, we have
$\Phi M((w_0w)^{-1},y)^{\bullet}\cong M(y^{-1}w_0,w)$.
Hence, by Proposition~\ref{prop1}\eqref{prop1-3}, the complex 
$M((w_0w)^{-1},y)^{\bullet}$ is concentrated in positions between
$-\mathbf{a}(w)$ and $\mathbf{a}(w)$.

Moreover, the complex $M((w_0w)^{-1},y)^{\bullet}$ consists of 
tilting modules, and the module $T(w_0)$ is a tilting module as well.
Hence, by \cite[Chapter~III(2),Lemma~2.1]{Ha}, the space
$\mathrm{Hom}_{\mathcal{D}^b(\mathcal{O})}(T(w_0),M((w_0w)^{-1},y)^{\bullet}
\llbracket i\rrbracket)$ can be computed already in the homotopy
category. However, if $i>\mathbf{a}(w)$, then from the previous
paragraph it follows that all nonzero components of the complex
$M((w_0w)^{-1},y)^{\bullet}\llbracket i\rrbracket$ are in negative 
positions. As $T(w_0)$ is in position zero, we obtain that the
morphism space from $T(w_0)$ to 
$M((w_0w)^{-1},y)^{\bullet}\llbracket i\rrbracket$ in the homotopy 
category is zero. This implies $\mathrm{Ext}_{\mathcal{O}}^i(T(w),L(y))=0$ 
for all $i>\mathbf{a}(w)$ and all $x\in W$ and the claim of the lemma follows.
\end{proof}

\begin{lemma}\label{lem17}
The projective dimension of $T(w)$
is at least $\mathbf{a}(w)$.
\end{lemma}

\begin{proof}
Let $\mathbf{R}$ and $\mathbf{L}$ denote the right and the left cells of 
$w$, respectively, and $\mathbf{T}\in\{\mathbf{R},\mathbf{L}\}$.
Recall (see for example the detailed explanation in \cite[Section~5]{Ka}) 
that 
\begin{displaymath}
\mathbf{a}(w)=
\min\{i\in\mathbb{Z}: \exists x\in \mathbf{T}\text{ s.t. }
\Hom_{\mathtt{A}}(P(x)\langle -i\rangle,P(e))\neq 0\}
\end{displaymath}
(such $x$ will be the Duflo involution in $\mathbf{T}$).
Using the Ringel self-duality of $\mathcal{O}_0$ 
(which accounts to the application of $\mathrm{T}_{w_0}$ followed 
by $\star$) and using \cite{AS} we obtain
\begin{equation}\label{eq3}
\mathbf{a}(w)=
\min\{i\in\mathbb{Z}: \exists x\in \mathbf{T}\text{ s.t. } 
\Hom_{\mathtt{A}}(T(w_0)\langle -i\rangle,T(w_0x))\neq 0\}.
\end{equation}

Consider modules $M(y^{-1}w_0,w)$, where $y\in W$. If
$y$ is such that $y^{-1}w_0$ runs through $\mathbf{R}$, then from
Corollary~\ref{cor75} it follows that the socles of the modules
$M(y^{-1}w_0,w)$ are 
$L(x)\langle -\mathbf{a}(w)\rangle$, $x\in\mathbf{R}$. 

Applying $\Phi^{-1}$ and using Theorem~\ref{thm14}, for $y$ as above
we obtain $\mathcal{M}(w^{-1}w_0,y)^{i}=0$, $i>\mathbf{a}(w)$, while
\begin{displaymath}
\mathcal{M}(w^{-1}w_0,y)^{\mathbf{a}(w)}\cong
T(w_0x^{-1}w_0)\langle\mathbf{a}(w)\rangle,
\end{displaymath}
where $x\in\mathbf{R}$ (\cite[Theorem~3.3]{Ma}). 
By \eqref{eq3} we can choose $y$ such that $y^{-1}w_0\in \mathbf{R}$
and for the corresponding $x$ we have
\begin{equation}\label{eq55}
\Hom_{\mathtt{A}}(T(w_0)\langle -\mathbf{a}(x^{-1}w_0)\rangle,T(w_0x^{-1}w_0))\neq 0. 
\end{equation}
At the same time all tilting summands of 
$\mathcal{M}(w^{-1}w_0,y)^{\mathbf{a}(w)-1}$ have (up to shift)
the form $T(z)$, where $z\leq_{LR} w_0x^{-1}w_0$. As $\mathbf{a}$
respects the two-sided order, we thus get
\begin{equation}\label{eq555}
\Hom_{\mathtt{A}}(T(w_0)\langle -\mathbf{a}(x^{-1}w_0)+1\rangle,T(z))= 0 
\end{equation}
for any such summand. From \eqref{eq55} it follows that there is a
nonzero morphism from $T(w_0)\langle\mathbf{a}(w)-\mathbf{a}(x^{-1}w_0)
\rangle$ to 
$\mathcal{M}(w^{-1}w_0,y)^{\bullet}\llbracket\mathbf{a}(w)\rrbracket$ in
the category of complexes. From \eqref{eq555} it follows that there are
no homotopy from $T(w_0)\langle\mathbf{a}(w)-\mathbf{a}(x^{-1}w_0)
\rangle$ to 
$\mathcal{M}(w^{-1}w_0,y)^{\bullet}\llbracket\mathbf{a}(w)\rrbracket$.
Hence there is  a nonzero
homomorphism from $T(w_0)\langle\mathbf{a}(w)-\mathbf{a}(x^{-1}w_0)\rangle$ to 
$\mathcal{M}(w^{-1}w_0,y)^{\bullet}\llbracket\mathbf{a}(w)\rrbracket$ 
in the homotopy category. Using \cite[Chapter~III(2),Lemma~2.1]{Ha}
and the adjunction, we thus get
\begin{displaymath}
\mathrm{Ext}_{\mathcal{O}}^{\mathbf{a}(w)}
(T(w_0),M(w^{-1}w_0,y))\cong
\mathrm{Ext}_{\mathcal{O}}^{\mathbf{a}(w)}(T(w),L(y))\neq 0.
\end{displaymath}
The claim of the lemma follows.
\end{proof}

\section{Projective dimension of indecomposable injective modules}\label{s4}

In this section we prove the second main result of this paper
(see \cite[Conjecture~15(b)]{Ma2}).

\begin{theorem}\label{thm18}
Let $w\in W$. Then the projective dimension of the module $I(w)$
equals $2\mathbf{a}(w_0w)$.
\end{theorem}

\begin{proof}
We prove this theorem by a descending induction with respect to the
two-sided order on $W$. Note that the projective dimension of $I(w)$
is an invariant of a two-sided cell by \cite[Theorem~11]{Ma2}. 
If $w=w_0$, the module $I(w_0)$ is projective and thus has 
projective dimension zero, which agrees with our claim. 

Fix $w\in W$, $w\neq w_0$, and assume that the claim of the theorem 
is true for all $x\in W$ such that $x>_{LR}w$.

\begin{lemma}\label{lem19}
The projective dimension of $I(w)$
is at most $2\mathbf{a}(w_0w)$.
\end{lemma}

\begin{proof}
Let $d$ be the Duflo involution in the right cell of $w$. As
projective dimension of $I(w)$ is an invariant of a two-sided cell 
by \cite[Theorem~11]{Ma2}, it is enough to prove the claim in the case
$w=d$. In this proof we consider all modules as ungraded.

Consider the injective module $\theta_{d}\theta_d I(e)$. We have
\begin{equation}\label{eq4}
\mathrm{Hom}_{\mathcal{O}}(L(d),\theta_{d}\theta_d I(e))\cong
\mathrm{Hom}_{\mathcal{O}}(\theta_{d}L(d),\theta_d I(e)).
\end{equation}
By Corollaries~\ref{cor3} and \ref{cor7} we have that 
$\theta_{d}L(d)$ has simple socle $L(d)$ and hence the homomorphism
space from $\theta_{d}L(d)$ to $\theta_d I(e)=I(d)$ is nonzero. From 
\eqref{eq4} it thus follows that $I(d)$ is a direct summand of 
$\theta_{d}\theta_d I(e)$. Hence to prove the statement
of the lemma it is enough to show that the projective dimension of
$\theta_{d}\theta_d I(e)$ is at most $2\mathbf{a}(w_0w)$.

For all $i\in\{0,1,\dots\}$ and all $y\in W$ by adjointness we have
\begin{displaymath}
\mathrm{Ext}^i_{\mathcal{O}}(\theta_{d}\theta_d I(e),L(y))\cong
\mathrm{Ext}^i_{\mathcal{O}}(\theta_d I(e),\theta_d L(y)).
\end{displaymath}
The module $\theta_d L(y)=M(d,y)$ is represented in the derived category
by the complex $\mathcal{M}(d,y)^{\bullet}$. From Theorem~\ref{thm14}
and Proposition~\ref{prop1} it follows that this complex is concentrated
between positions $-\mathbf{a}(w_0d)$ and $\mathbf{a}(w_0d)$.
To proceed we need the following generalization of this observation:

\begin{lemma}\label{lemnewnew}
For any  $X\in \mathcal{O}$ the module $\theta_d X$ is represented in the
derived category by some complex of tilting modules, concentrated 
between positions $-\mathbf{a}(w_0d)$ and $\mathbf{a}(w_0d)$.
\end{lemma}

\begin{proof}
We prove the claim by induction on the length of $X$. For simple
$X$ the claim follows from the last paragraph before this lemma.
To prove the induction step we consider a short exact sequence
\begin{displaymath}
0\to Y \to X \to Z \to 0 
\end{displaymath}
such that $Z$ is simple. Applying $\theta_d$ we get a short exact sequence
\begin{equation} \label{eq777}
0\to \theta_d Y \to \theta_d X \to \theta_d Z \to 0. 
\end{equation}
If $\theta_d Z=0$, then $\theta_d X=\theta_d Y$
and the claim follows from the inductive assumption. 
Otherwise, from the inductive assumption we
have a complex $\mathcal{C}^{\bullet}$ of tilting modules, concentrated 
between positions $-\mathbf{a}(w_0d)$ and $\mathbf{a}(w_0d)$, which
represents $\theta_d Y$. As $Z$ is simple, from the basis of the induction
we have a complex $\mathcal{B}^{\bullet}$ of tilting modules, concentrated 
between positions $-\mathbf{a}(w_0d)$ and $\mathbf{a}(w_0d)$, which
represents $\theta_d Z$. The extension given by 
\eqref{eq777} corresponds to some
morphism from $\mathcal{B}^{\bullet}\llbracket -1\rrbracket$ to
$\mathcal{C}^{\bullet}$ in the homotopy category
(\cite[Chapter~III(2),Lemma~2.1]{Ha}). Taking the cone of 
this morphism we get a complex of tilting modules, concentrated 
between positions $-\mathbf{a}(w_0d)$ and $\mathbf{a}(w_0d)$, which
represents $\theta_d X$. This completes the proof.
\end{proof}

By Lemma~\ref{lemnewnew}, we have that $\theta_d I(e)$ is represented 
in the derived category by some complex of tilting modules, 
concentrated between positions 
$-\mathbf{a}(w_0d)$ and $\mathbf{a}(w_0d)$.

Now for all $i\in\{0,1,\dots\}$ we have
\begin{displaymath}
\mathrm{Ext}^i_{\mathcal{O}}(\theta_d I(e),\theta_d L(y))\cong
\mathrm{Hom}_{\mathcal{D}^b(\mathcal{O})}
(\theta_d I(e),\theta_d L(y)\llbracket i\rrbracket).
\end{displaymath}
If we represent both $\theta_d I(e)$ and $\theta_d L(y)$ by the corresponding
complexes of tilting modules, then the latter morphism space can be computed
already in the homotopy category (\cite[Chapter~III(2),Lemma~2.1]{Ha}).
However, since both complexes are concentrated 
between positions $-\mathbf{a}(w_0d)$ and $\mathbf{a}(w_0d)$,
it follows that for 
$i>2\mathbf{a}(w_0d)$ the corresponding space in the homotopy category is
zero. The claim of the lemma follows.
\end{proof}

\begin{lemma}\label{lem20}
The projective dimension of $I(w)$
is at least $2\mathbf{a}(w_0w)$.
\end{lemma}

\begin{proof}
We again prove the statement in the case $w=d$ (the Duflo involution) 
and work with ungraded modules. Let $X$ denote the cokernel of the natural
inclusion $L(d)\hookrightarrow I(d)\cong \theta_d I(e)$. 
Consider the short exact sequence
\begin{equation}\label{eq8}
0\to \theta_d L(d)\to \theta_d\theta_d I(e)
\to  \theta_d X\to 0.
\end{equation}
From the proof of Lemma~\ref{lem19} we know that the module 
$\theta_d\theta_d I(e)$ contains $I(d)$ as a direct summand.
From \cite[(1)]{Ma2} it follows that the module
$\theta_d\theta_d I(e)$ is a direct sum of injective modules
$I(x)$, where $x\geq_{LR}d$.  If $x>_{L}d$, then from the 
inductive assumption we know that the projective dimension of 
$I(x)$ is strictly less that $2\mathbf{a}(w_0w)$. Hence to prove the 
statement of the lemma it is enough to show that the projective 
dimension of $\theta_{d}\theta_d I(e)$ is at least $2\mathbf{a}(w_0w)$.
To prove this it is enough to show that
\begin{equation}\label{eq5}
\mathrm{Ext}^{2\mathbf{a}(w_0w)}_{\mathcal{O}}
(\theta_d\theta_d I(e),\theta_d L(d))\neq 0.
\end{equation}

The module $\theta_d L(d)$ is self-dual by 
Proposition~\ref{prop1}\eqref{prop1-1}. It is also represented in the
derived category by the complex $\mathcal{M}(d,d)^{\bullet}$, which
is a linear complex of tilting modules and hence does not have trivial direct 
summands. By Theorem~\ref{thm14} and Proposition~\ref{prop1} we know
that the leftmost nonzero position in $\mathcal{M}(d,d)^{\bullet}$
is $-\mathbf{a}(w_0d)$. Hence, by \cite[Lemma~6 and Corollary~1]{MO}
we have 
\begin{equation}\label{eq6}
\mathrm{Ext}^{2\mathbf{a}(w_0w)}_{\mathcal{O}}
(\theta_d L(d),\theta_d L(d))\neq 0.
\end{equation}

Using Lemma~\ref{lemnewnew} one shows  that 
\begin{displaymath}
\mathrm{Ext}^{i}_{\mathcal{O}}
(\theta_d X,\theta_d L(d))= 0
\end{displaymath}
for all $i>2\mathbf{a}(w_0w)$. Hence, applying
$\mathrm{Hom}_{\mathcal{O}}({}_-,\theta_d L(d))$ to the 
short exact sequence \eqref{eq8} and going to
the long exact sequence in homology we obtain a surjection
\begin{displaymath}
\mathrm{Ext}^{2\mathbf{a}(w_0w)}_{\mathcal{O}}
(\theta_d\theta_d I(e),\theta_d L(d))\tto
\mathrm{Ext}^{2\mathbf{a}(w_0w)}_{\mathcal{O}}
(\theta_d L(d),\theta_d L(d)).
\end{displaymath}
Now \eqref{eq5} follows from \eqref{eq6} and completes the proof.
\end{proof}
\end{proof}

Above in the paper we have seen that the modules $M(x,y)$, $x,y\in W$, 
play important role in the combinatorics of the category $\mathcal{O}$.
From this point of view the following problem looks rather natural:

\begin{problem}\label{prob30}
{\rm 
Determine the projective dimension of 
$M(x,y)$ for all $x,y\in W$.
}
\end{problem}

If $x=e$, then $M(x,y)$ is the simple module $L(y)$ and has projective 
dimension $2l(w_0)-l(w)$ (\cite[Proposition~6]{Ma2}). If $y=w_0$, then 
$M(x,y)$ is the tilting module $T(w_0x)$ and has projective dimension
$\mathbf{a}(w_0x)$ (Theorem~\ref{thm15}). In the general case I do not 
have any conjectural formula for the projective dimension of
$M(x,y)$.

\vspace{0.5cm}

\noindent 
Department of Mathematics, Uppsala University, SE-751 06, Uppsala, SWEDEN,
e-mail: {\small \tt mazor@math.uu.se},\\ 
web: http://www.math.uu.se/$\tilde{\hspace{1mm}}$mazor/
\vspace{0.5cm}

\end{document}